\documentclass{amsart}

\usepackage{amssymb,amsmath,latexsym,amscd}

\usepackage{pstricks,pst-text,pst-grad,pst-node,pst-tree,pst-plot,pst-3dplot,pst-barcode}
\usepackage{pstricks-add}
\usepackage{pst-solides3d}
\usepackage[tiling]{pst-fill}
\usepackage{graphicx}
\usepackage{fontenc}%
\usepackage{marvosym}
\usepackage{paralist}
\usepackage{color}

\newtheorem{theorem}{Theorem}[section]
\newtheorem{lemma}[theorem]{Lemma}
\newtheorem{corollary}[theorem]{Corollary}

\theoremstyle{definition}

\theoremstyle{remark}
\newtheorem{remark}[theorem]{Remark}

\newcommand{\norm}[1]{\left\|#1\right\|}

\numberwithin{equation}{section}

\begin{document}

\title[Klee-Phelps Convex Groupoids]{Klee-Phelps Convex Groupoids}

\author[J.F. Peters]{J.F. Peters}
\email{James.Peters3@umanitoba.ca}
\address{Computational Intelligence Laboratory,
University of Manitoba, WPG, MB, R3T 5V6, Canada}
\thanks{The research has been supported by the Scientific and Technological Research
Council of Turkey (T\"{U}B\.{I}TAK) Scientific Human Resources Development
(BIDEB) under grant no: 2221-1059B211402463 and Natural Sciences \&
Engineering Research Council of Canada (NSERC) discovery grant 185986.}
\author[M. A. \"{O}zt\"{u}rk]{M.A. \"{O}zt\"{u}rk}
\email{maozturk@adiyaman.edu.tr}
\author[M. U\c{c}kun]{M. U\c{c}kun}
\email{muckun@adiyaman.edu.tr}
\address{Department of Mathematics, Faculty of Arts and Sciences, Ad\i yaman University, Ad\i yaman, Turkey.}


\date{}

\dedicatory{}

\commby{Pham Huu Tiep}

\begin{abstract}
We prove that a pair of proximal Klee-Phelps convex groupoids $A(\circ)$,$B(\circ)$ in a finite-dimensional normed linear space $E$ are normed proximal, {\em i.e.}, $A(\circ)\ \delta\ B(\circ)$ if and only if the groupoids are normed proximal.   In addition, we prove that the groupoid neighbourhood $N_z(\circ)\subseteq S_z(\circ)$ is convex in $E$ if and only if $N_z(\circ) = S_z(\circ)$.
\end{abstract}


\maketitle

\section{Introduction}
Klee-Phelps groupoids are named after V.L. Klee and R.R. Phelps.  It was Klee, who characterized a convex set in terms of a subset of a finite-dimensional Euclidean space that is the set of all points that are nearest a point in the space~\cite{Klee1949}.  It was Phelps, who proved that for a subset $S$ in a inner product space $E$, the set $S_z$ of all points having $z\in S$ as the nearest point is a convex set~\cite{Phelps1957}.  Let $S_z$ be the set of all points in $E$ having $z\in S$ as the nearest point in $S$, defined by
\[
S_z = \left\{x\in E: \left\|x - z\right\| = \mathop{inf}\limits_{y\in S}\left\|x - y\right\|\right\}.
\]

\begin{lemma}\label{PhelpsL} {\rm (Phelps' Lemma~\cite{Phelps1957})} If $E$ is an inner product space, $S\subset E$ and $z\in S$, then $S_{z}$ is convex.
\end{lemma}

The main result in this paper is that a pair of Klee-Phelps convex groupoids $A(\circ),B(\circ)$ in a finite-dimensional normed linear space are proximal if and only if $A(\circ),B(\circ)$ are normed proximal.
%
\section{Preliminaries}

Let $E$ be a finite-dimensional real normed linear space, $A,B\subset E$, $x,y\in E$. The Hausdorff distance
$D(x,A)$ is defined by $D(x,A)=inf\left\{\norm{x-y}: y\in
A\right\}  $ and $\norm{x-y}$ is the distance between vectors
$x$ and $y$. The \u{C}ech closure~\cite{Cech1966} of $A$ (denoted by $\mbox{cl}A$) is
defined by $\mbox{cl}A=\left\{  x\in V:D(x,A)=0\right\}  $. The sets $A$ and $B$ are
proximal (near) (denoted $A\ \delta\ B$), provided $\mbox{cl}A\ \cap\ \mbox{cl}B\neq
\emptyset$~\cite{Concilio2009,Naimpally2012}.  A nonempty space endowed with a proximity relation is
called a proximity space~\cite{Peters2012ams}.  The space $E$ endowed with the proximity
relation $\delta$ is called a proximal linear space.  The assumption made here is that each proximal linear space is a topological space that provides the structure needed to define proximity relations.  The proximity relation $\delta$ defines a nearness relation between convex groupoids useful in many applications.  
A subset $K\subset E$ is \emph{convex}, provided, for every pair points $x,y\in K$, the line segment $\overline{xy}$ connecting $x$ and $y$ belongs to $K$.

Let $S\subset E, x,y\in S$, and $S_x,S_y$ are nonempty Klee-Phelps nearest point sets.  From Lemma~\ref{PhelpsL}, $S_x,S_y$ are convex sets.  
Sets $S_x,S_y$ are proximal if and only if $\norm{a-b} = 0$ for some $a\in \mbox{cl}S_x, b\in \mbox{cl}S_y$.  That is, convex sets $S_x,S_y$ are near, provided the convex set $S_x$ has at least one point in common with $S_y$.  In effect, $S_x,S_y$ are proximal if and only if $cl(S_x)\  \cap \ cl(S_y)\neq \emptyset$.
A \textit{convex hull} of a subset $A$ in $E$ is the smallest convex set that contains $A$.

\begin{lemma}\label{lem:hull} 
If $\mbox{cl}(S_{z})$ is a convex hull in $E$, then $S_{z}\subseteq \mbox{cl}S_{z}$.
\end{lemma} 
\begin{proof}
Let $x\in (X\setminus S_{z})$ such that $x = y$ for some $y\in \mbox{cl}(S_{z})$. Consequently, $x\in \mbox{cl}(S_{z})$.  Hence, $S_{z}\subseteq \mbox{cl}(S_{z})$.
\end{proof} 

Let $S\subset E, z\in S, \varepsilon \mathrel{\mathop :}= \mathop{inf}\limits_{y\in S}\norm{x - y}$.
The neighborhood of a point in a Klee-Phelps convex set (briefly, $N_z$) is defined by
\[
N_{z,\varepsilon} = \left\{x\in E: \norm{x-z} \leq \varepsilon\right\}.
\]

\begin{lemma}\label{lem1} If $S_z$ is a Klee-Phelps convex set in $E$, $S\subset E, z\in S$, then $N_{z,\varepsilon} \cap \mbox{cl}S_z\neq \emptyset$.
\end{lemma}
\begin{proof}
For each $x\in S_z$, $\norm{x - z} = \mathop{inf}\limits_{y\in S}\norm{x - y}$.  Consequently, $x\in N_{z,\varepsilon}$.  Hence, $N_{z,\varepsilon} \cap \mbox{cl}S_z\neq \emptyset$.
\end{proof}

\begin{corollary}
If $S_z$ is a Klee-Phelps convex set in $E$, $S\subset E, z\in S$, then $N_{z,\varepsilon}\ \delta\ S_z$.
\end{corollary}
\begin{proof}
Immediate from Lemma~\ref{lem1} and the definition of proximity\ $\delta$.
\end{proof}

\begin{theorem}\label{thm:convexNbd}
Let $S_z$ be a Klee-Phelps convex set in $E$, $S\subset E, z\in S$.  $N_{z,\varepsilon}\subseteq S_z$ is convex if and only if
$N_{z,\varepsilon} = S_z$.
\end{theorem}
\begin{proof}
$N_{z,\varepsilon}\subseteq S_z\Leftrightarrow$ $\norm{x - z} = \mathop{inf}\limits_{y\in S}\norm{x - y}$ for each $x\in S_z\Leftrightarrow N_{z,\varepsilon} = S_z$.  Hence, $N_z$ is convex.
\end{proof}

A groupoid is a system $S(\circ)$ that consists of a nonempty set $S$ together with a
binary operation $\circ$ on $S$~\cite{Clifford1964}.  A Klee-Phelps groupoid is a system $S_z(\circ)$ that consists of a nonempty convex set $S_z\subset E$ together with a binary operation $\circ$ on $S_z$ such that $S_z(\circ)\subseteq S_z$.  

\begin{corollary}
Let $S_z(\circ)$ is a Klee-Phelps convex groupoid in $E$, $S\subset E, z\in S$.  The neighbourhood groupoid $N_{z,\varepsilon}(\circ)\subseteq S_z(\circ)$ is convex if and only if $N_{z,\varepsilon}(\circ) = S_z(\circ)$.
\end{corollary}
\begin{proof}
Immediate from Theorem~\ref{thm:convexNbd}.
\end{proof}

\begin{theorem}\label{thm:nearConvexity}
Let $U,V$ be proximal linear spaces, $z\in S\subset U, z^{\prime}\in S^{\prime}\subset V$ and let $S_z\left(\circ\right), S^{\prime}_{z^{\prime}}\left(\circ\right)$ be Klee-Phelps convex groupoids. Then 
\[
\mbox{cl}\left(
S_{z}\right) \cap \mbox{cl}\left(S^{\prime}_{z^{\prime}}\right)\neq\emptyset \Leftrightarrow S_{z}\left(\circ\right)\ \delta\ S^{\prime}_{z^{\prime}}\left(\circ\right).
\]
\end{theorem}
\begin{proof}
Immediate from the definition of the proximity\ $\delta$.
\end{proof}

\section{Main Results}

Let $U^{m}$ and $V^{n}$ be m- and n-dimensional proximal linear spaces, respectively, $m,n\in\mathbb{N}$.  Also, let $\mathcal{S},T,s,t$ be the mappings given in 

\begin{align*}
\begin{CD}
U^m\times U^m @>\text{$\mathcal{S}$}>> U^k, &\qquad U^m @>\text{$s$}>> U^k,\\
V^n\times V^n @>\text{$T$}>> V^k, &\qquad V^n @>\text{$t$}>> V^k.
\end{CD}
\end{align*}

\vspace{3mm}




\begin{lemma}
\label{lem2}Let $U^{m},V^{n}$ be proximal linear spaces, $y,z\in S\subset U^{m},y^{\prime},z^{\prime}\in S^{\prime}\subset U^{m}$, $N_{y,\varepsilon}
,S_{z}\subset U^{m}$ and $N_{y^{\prime},\varepsilon},S^{\prime}_{z^{\prime}}\subset V^{n}$.
\begin{compactenum}[1$^o$]
\item If $m=n$ and $N_{y,\varepsilon}$ $\delta$ $N_{y^{\prime},\varepsilon}$, then $S_{z}$ $\delta$ $S^{\prime}_{z^{\prime}}$.

\item If $m\neq n$ and $s\left(N_{y,\varepsilon}\right)$ $\delta$ $t\left(N_{y^{\prime},\varepsilon}\right)$,
then $s\left(S_{z}\right)$ $\delta$ $t\left(S^{\prime}_{z^{\prime}}\right)$.

\item If $m\neq n$ and $\mathcal{S}\left(N_{y,\varepsilon}\times N_{y,\varepsilon}\right)$ $\delta$ $T\left(
N_{y^{\prime},\varepsilon}\times N_{y^{\prime},\varepsilon}\right)$ \\ 
then $\mathcal{S}\left(S_{z}\times S_{z}\right)$ $\delta$ $T\left(S^{\prime}_{z^{\prime}}\times S^{\prime}_{z^{\prime}}\right)$.
\end{compactenum}
\end{lemma}

\begin{proof}\mbox{}
\begin{compactenum}[1$^o$]
\item Let $m=n$ and $N_{y,\varepsilon}$ $\delta$ $N_{y^{\prime},\varepsilon}$. Then 
$\mbox{cl}\left(N_{y,\varepsilon}\right)$ $\cap$ $\mbox{cl}\left(N_{y^{\prime},\varepsilon}\right)\neq\emptyset$. 
From Lemma~\ref{lem1}, $N_{y,\varepsilon}\subseteq \mbox{cl}\left(S_{z}\right)$ and 
$N_{y^{\prime},\varepsilon}\subseteq \mbox{cl}\left(^{\prime}S_{z^{\prime}}\right)$. 
Consequently, $\mbox{cl}\left(S_{z}\right)$\  $\cap$\ $\mbox{cl}\left(S_{z^{\prime}}\right)
\neq\emptyset$. Hence, $S_{z}$ $\delta$ $S^{\prime}_{z^{\prime}}$.

\item Let $m\neq n$ and $s\left(N_{y,\varepsilon}\right)$ $\delta$ $t\left(N_{y^{\prime},\varepsilon}\right)$.
Then $\mbox{cl}\left(s\left(N_{y,\varepsilon}\right)\right)$ $\cap$ 
$\mbox{cl}\left(t\left(N_{y^{\prime}}\right)\right)  \neq\emptyset$. From Lemma~\ref{lem1},
$N_{y,\varepsilon}\subseteq \mbox{cl}\left(S_{z}\right)$ and
$N_{y^{\prime},\varepsilon}\subseteq \mbox{cl}\left(S_{z^{\prime}}\right)$.
Consequently, $\mbox{cl}\left(s\left(S_{z}\right)\right)  \cap \mbox{cl}\left(
t\left(S^{\prime}_{z^{\prime}}\right)\right)$  $\neq\emptyset$.  Hence, $s\left(
S_{z}\right)$ $\delta$ $t\left(S^{\prime}_{z^{\prime}}\right)$.

\item Let $m\neq n$ and $\mathcal{S}\left(N_{y,\varepsilon}\times N_{y,\varepsilon}\right)$ $\delta$ $T\left(
(N_{y^{\prime},\varepsilon}\times (N_{y^{\prime},\varepsilon}\right)$. Then $\mbox{cl}\left(\mathcal{S}\left(
N_{y,\varepsilon}\times N_{y,\varepsilon}\right)  \right)  $ $\cap$ 
$\mbox{cl}\left(T\left(N_{y^{\prime},\varepsilon}\times N_{y^{\prime},\varepsilon}\right)\right)  \neq\emptyset$. 
From Lemma \ref{lem1}, $N_{y,\varepsilon}\subseteq \mbox{cl}\left(  S_{z}\right)  $ and 
$N_{y^{\prime},\varepsilon}\subset \mbox{cl}\left(S^{\prime}_{z^{\prime}}\right)
$. Consequently, $\mbox{cl}\left(  \mathcal{S}\left(S_{z}\times S_{z}\right)  \right)  \cap
cl\left(  T\left(  S^{\prime}_{z^{\prime}}\times S^{\prime}_{z^{\prime}}\right)  \right)
\neq\emptyset$. Hence, $\mathcal{S}\left(  S_{z}\times S_{z}\right)  $
$\delta$ $T\left(  S^{\prime}_{z^{\prime}}\times S^{\prime}_{z^{\prime}}\right)  $.
\end{compactenum}
\end{proof}

\begin{theorem}
Let $U^{m},V^{n}$ be proximal linear spaces, $N_{y,\varepsilon}\left(\circ\right)\subset U^{m}$,
$N_{y^{\prime},\varepsilon}\left(\circ\right)
\subset V^{n}$ proximal neighborhood groupoids and let $S_{z}\left(\circ\right)
\subset U^{m}$ $,S^{\prime}_{z^{\prime}}\left(\circ\right) \subset V^{n}$ be proximal
Klee-Phelps convex groupoids.
\begin{compactenum}[1$^o$]
\item If $m=n$ and $N_{y,\varepsilon}\left(\circ\right)$\ $\delta$\ $N_{y^{\prime},\varepsilon}\left(\circ\right)$ 
then $S_{z}\left(\circ\right)$ $\delta$ $S^{\prime}_{z^{\prime}}\left(\circ\right)$.

\item If $m\neq n$ and $s\left(N_{y,\varepsilon}\left(\circ\right)\right)$ $\delta$ 
$t\left(N_{y^{\prime},\varepsilon}\left(\circ\right)\right)$
then $s\left(S_{z}\left(\circ\right)  \right)$ $\delta$ $t\left(
S^{\prime}_{z^{\prime}}\left(\circ\right)\right)$.

\item If $m\neq n$ and \small{$\mathcal{S}\left(N_{y,\varepsilon}\left(\circ\right)\times 
N_{y,\varepsilon}\left(\circ\right)\right)$ $\delta$
$T\left(N_{y^{\prime},\varepsilon}\left(\circ\right)\times N_{y^{\prime},\varepsilon}\left(\circ\right)\right) $ } \normalsize 
{then  $\mathcal{S}\left(  S_{z}\left(  \circ\right)
\times S_{z}\left(  \circ\right)  \right)  $ $\delta$ $T\left(  S^{\prime}_{z^{\prime}%
}\left(  \circ\right)  \times S^{\prime}_{z^{\prime}}\left(  \circ\right)  \right)  $.}
\end{compactenum}
\end{theorem}

\begin{proof}
Immediate from Lemma \ref{lem2}.
\end{proof}


\begin{remark}
Let $A,B\subset E, A\ \delta\ B$, provided $\norm{a-b}$ for some $a\in \mbox{cl}A,
b\in \mbox{cl}B$, {\em i.e.}, $\mbox{cl}A\cap \mbox{cl}B\neq\emptyset$.  From the definition
of $N_{z,\varepsilon}$, a neighborhood of point $z\in E$, we know $\norm{x-z} < \varepsilon$ for each
$x\in E$ that is in $N_{z,\varepsilon}$.  If a pair of neighborhoods $N_{z,\varepsilon},N_{z^{\prime},\varepsilon}$ are proximal,
then the neighborhoods have at least one point $x$ in common.  For this reason, we can then
write $\norm{x-z} = \norm{x-z^{\prime}}$.  This leads to what is known as normed proximity $\delta_P$,
{\em i.e.}, $A\ \delta_P\ B$, provided $\norm{x-z} = \norm{x-z^{\prime}}$ for some $x\in E, z\in A, z^{\prime}\in B$.  \qquad \textcolor{black}{$\blacksquare$}
\end{remark}
Let $\left(E,\delta,\delta_P\right)$ denote a finite-dimensional normed linear space endowed with proximities $\delta,\delta_P$ (briefly, proximal linear space).
\begin{lemma}
\label{lem3} Let $\left(E,\delta,\delta_P\right)$ be a proximal linear space, $A,B\subset E$.
\begin{compactenum}[(i)]
\item If $A = N_{z,\varepsilon}$, $B = N_{z^{\prime},\varepsilon}$, then $A\ \delta_{P}\ B$ $\Rightarrow$ $A\ \delta\ B$.

\item If $A = S_{z}$, $B = S_{z^{\prime}}$ then $A\ \delta_{P}\ B$ $\Leftrightarrow
$ $A\ \delta\ B$.
\end{compactenum}
\end{lemma}

\begin{proof}$\mbox{}$
\begin{compactenum}[(i)]
\item $A = N_{z,\varepsilon}\subset E$, $B = N_{z^{\prime},\varepsilon}\subset E$.

\begin{tabular}
[c]{ll}%
$A\ \delta_{P}\ B$ & $\Rightarrow\left\Vert x-z\right\Vert =\left\Vert
x-z^{\prime}\right\Vert ,$ for some $x\in E, z\in A\in E, z^{\prime}\in B$\\
& $\Rightarrow x\in \mbox{cl}A$ and $x\in \mbox{cl}B$\\
& $\Rightarrow \mbox{cl}A\cap \mbox{cl}B\neq\emptyset$\\
& $\Rightarrow A\ \delta\ B$.
\end{tabular}

\item Let $A=S_{z}\subset E$ and $B=S^{\prime}_{z^{\prime}}\subset E$.

\begin{tabular}
[c]{ll}%
$A\ \delta\ B$ & $\Leftrightarrow clA\cap clB\neq\emptyset$\\
& $\Leftrightarrow\exists x\in clA\cap clB$\\
& $\Leftrightarrow x\in clA$ and $x\in clB$\\
& $\Leftrightarrow\left\Vert x-z\right\Vert =\left\Vert x-z^{\prime
}\right\Vert $ for some $x\in E, z\in A,z^{\prime}\in B$\\
& $\Leftrightarrow A\ \delta_{P}\ B$.
\end{tabular}
\end{compactenum}
\end{proof}

\begin{theorem}
Let $\left(E,\delta,\delta_P\right)$ be a proximal linear space, $A,B,S,S^{\prime}\subset E$, $z\in S, z^{\prime}\in S^{\prime}$.
\begin{compactenum}[(i)]
\item If $A({\circ}) = N_{z,\varepsilon}({\circ}), B({\circ}) = N_{z^{\prime},\varepsilon}({\circ})$,
then $A({\circ})\ \delta_P\ B({\circ})$ $\Rightarrow$ $A({\circ})\ \delta\ B({\circ})$.

\item If $A({\circ}) = S_{z}\left(\circ\right), B({\circ}) = S^{\prime}_{z^{\prime}}\left(  \circ\right)$, 
then $A({\circ})\ \delta_{P}\ B({\circ})$ $\Leftrightarrow$ $A({\circ})\ \delta\ B({\circ})$.
\end{compactenum}
\end{theorem}

\begin{proof}
Immediate from Lemma \ref{lem3}.
\end{proof}





\end{document}